\documentclass[12pt,a4paper]{amsart} %
\usepackage[
  a4paper,
  top=2.8cm,
  bottom=3.5cm,
  left=2.2cm,
  right=2.2cm,
  heightrounded
]{geometry}
\usepackage[skip=0.8em plus 0.2em minus 0.1em]{parskip}
\usepackage{amsmath,amsfonts,amssymb,amsthm}

\usepackage[colorlinks=true,linkcolor=blue,citecolor=blue,urlcolor=blue]{hyperref}

\allowdisplaybreaks

\theoremstyle{plain}
\newtheorem{question}{Question}

\newtheorem{theorem}{Theorem}

\newtheorem{corollary}{Corollary}
\newtheorem{lemma}{Lemma} 

\newtheorem{theoremx}{Theorem}

\newtheorem{definitionx}{Definition}

\newtheorem{lemmax}{Lemma}

\theoremstyle{remark}
\newtheorem{remark}{Remark}

\title{A discrete Hardy uncertainty principle}
\author{Torgeir Keun Lysen}
\email{torgeir.lysen@gmail.com}

\begin{document}
\begin{abstract}
    We show that knowing the decay of a function $f$ on a discrete set $\Lambda\subset\mathbb{R}$ and the decay of its Fourier transform $\hat{f}$ on a discrete set $M\subset\mathbb{R}$ is enough to determine the global decay of $f$ and $\hat{f}$, provided that $(\Lambda,M)$ is a supercritical pair in the sense of Kulikov, Nazarov, and Sodin. This decay transfer result leads to a discrete generalization of Morgan's uncertainty principle: it is enough to require $|f(\lambda)|\lesssim e^{-\frac{2}{p}A\pi|\lambda|^p}$ for all $\lambda\in\Lambda$ and $|\hat{f}(\mu)|\lesssim e^{-\frac{2}{q}A\pi|\mu|^q}$ for all $\mu\in M$, where $(p,q)$ are H\"{o}lder conjugates, $A>|\cos(\frac{r\pi}{2})|^\frac{1}{r}$, and $r:=\min\{p,q\}$. For $A=1$ and $p,q=2$, we also show that any such function must be a scaled Gaussian. This yields a discrete version of Hardy's uncertainty principle and resolves two questions posed by Ramos and Sousa.
\end{abstract}
\maketitle

\section{Introduction}
In this paper, we determine when Hardy's uncertainty principle and similar uncertainty principles can be recovered from knowledge of $|f|$ and $|\hat f|$ only on discrete subsets of $\mathbb{R}$. After the remarkable work of Radchenko and Viazovska on Fourier interpolation~\cite{Radchenko2019}, there has been substantial work on identifying how often a function can vanish in both time and frequency~\cite{ramos2022fourier,kulikov2025,adve2023density,lysen2025criticalasymmetricfourieruniqueness}. This is generally referred to as the study of Fourier uniqueness pairs and can be viewed as a discrete form of the uncertainty principle.
\begin{definitionx}[Fourier uniqueness pair]\label{def:fourier-uniqueness-pair}
    We call a pair $(\Lambda, M)$ with $\Lambda,M\subset \mathbb{R}$ a \emph{Fourier uniqueness pair} for a space $X\subset L^1(\mathbb{R})\cap L^2(\mathbb{R})$ if, for every $f\in X$,
    \[
    f|_\Lambda=0\quad\text{and}\quad \hat{f}|_M=0\quad\Longrightarrow\quad f\equiv 0.
    \]
    A pair that is not a \emph{Fourier uniqueness pair} is called a \emph{Fourier non-uniqueness pair}.
\end{definitionx}
We use the following normalization of the Fourier transform:
\[
\hat{f}(\xi):=\int_\mathbb{R}f(x)e^{-2\pi ix\xi}\,dx.
\]
Unless stated otherwise, we assume that the sequences $\Lambda,M\subset\mathbb{R}$ are ordered and unbounded in both directions.

We particularly note a result of Kulikov, Nazarov, and Sodin, who identified a broad class of Fourier uniqueness and non-uniqueness pairs called supercritical and subcritical pairs~\cite{kulikov2025}.
\begin{definitionx}[Supercritical and subcritical pairs]\label{def:density-criticality}
Given $1<p,q<\infty$ with $\frac{1}{p}+\frac{1}{q}=1$, we call a pair $(\Lambda, M)$ \emph{supercritical} if
\begin{align*}
        \limsup_{j\to\pm\infty}|\lambda_j|^{p-1}(\lambda_{j+1}-\lambda_j)\leq\overline{\alpha}\quad\text{and}\quad\limsup_{j\to\pm\infty}|\mu_j|^{q-1}(\mu_{j+1}-\mu_j)\leq \overline{\beta}
\end{align*}
with $\overline{\alpha}^\frac{1}{p}\overline{\beta}^\frac{1}{q}< \frac{1}{2}$. Conversely, the pair is \emph{subcritical} if there exists $\underline{\alpha}^\frac{1}{p}\underline{\beta}^\frac{1}{q}> \frac{1}{2}$ such that
\begin{align*}
        \liminf_{j\to\pm\infty}|\lambda_j|^{p-1}(\lambda_{j+1}-\lambda_j)\geq\underline{\alpha}\quad\text{and}\quad\liminf_{j\to\pm\infty}|\mu_j|^{q-1}(\mu_{j+1}-\mu_j)\geq \underline{\beta}.
\end{align*}

\end{definitionx}
We use $\mathcal{S}$ to denote the Schwartz space and $\mathcal{H}$ to denote the Fourier-symmetric Sobolev space
\[
\mathcal{H}:=\left\{f\in L^2(\mathbb{R}): \int_\mathbb{R}(1+x^2)|f(x)|^2\,dx+\int_\mathbb{R}(1+\xi^2)|\hat{f}(\xi)|^2\,d\xi<\infty\right\}.
\]
\begin{theoremx}[Kulikov--Nazarov--Sodin~\cite{kulikov2025}]\label{thm:kns-uniqueness}
    \textit{Suppose that $1 < p, q < \infty$, $\tfrac{1}{p} + \tfrac{1}{q} = 1$. Then}
\begin{itemize}
    \item[(i)] \textit{any supercritical pair $(\Lambda, M)$ is a uniqueness pair for $\mathcal{H}$;}
    \item[(ii)] \textit{any subcritical pair $(\Lambda, M)$ is a non-uniqueness pair for $\mathcal{S}$.}
\end{itemize}
\end{theoremx}

Recently, Ramos and Sousa proved that, for some of these supercritical pairs, it is enough to assume Gaussian decay only on the discrete sets of the pair.
\begin{theoremx}[Theorem~3 in~\cite{ramos2025pauli}]\label{thm:discrete-hardy}
    Let $A>1$ be fixed. There is $C_A>0$ such that the following holds. Suppose $\Lambda,M\subset\mathbb{R}$ are two discrete sets satisfying
    \begin{equation}\label{eq:condition-discrete-hardy}
        \max\left\{\limsup_{j\to\pm\infty}|\lambda_j|(\lambda_{j+1}-\lambda_j),\,\limsup_{j\to\pm\infty}|\mu_j|(\mu_{j+1}-\mu_j)\right\}<\frac{1}{C_A}.
    \end{equation}
    Then there are no non-zero functions $f\in\mathcal{H}$ such that
    \[
    \sup_{\lambda\in\Lambda}|f(\lambda)|e^{A\pi|\lambda|^2}+\sup_{\mu\in M}|\hat{f}(\mu)|e^{A\pi|\mu|^2}<\infty.
    \]
\end{theoremx}
This can be viewed as a discrete version of Hardy's uncertainty principle, where one only requires the function and its Fourier transform to decay like a Gaussian on discrete sets of points.
We will use the following version of Hardy's uncertainty principle~\cite{hardyuncertainty} given in~\cite[Chapter~8, Theorem~6]{levinentirefunctions}.
\begin{theoremx}[Hardy's uncertainty principle]\label{thm:hardy}
Let $N\geq0$ and $f \in L^{2}(\mathbb{R})$. Assume that
\[
\sup_{x\in\mathbb{R}} |f(x)| (1 + |x|)^{-N}e^{\pi |x|^{2}} +\sup_{\xi\in\mathbb{R}}|\hat{f}(\xi)|(1 + |\xi|)^{-N} e^{\pi |\xi|^{2}} <\infty.
\]
Then $f(x) = P(x) e^{-\pi x^{2}}$, where $P$ is a polynomial and $\deg P\leq N$.
\end{theoremx}
The proof of Theorem~\ref{thm:discrete-hardy} only applies to $A>1$ and yields the bound $C_A\sim (A-1)^{-3}$ as $A\to1^+$. The non-uniqueness part of Theorem~\ref{thm:kns-uniqueness} implies that the statement fails for $C_A<2$, and it clearly does not hold for $A<1$. In light of this, the following question arises.
\begin{question}[Question 24 in~\cite{ramos2025pauli}]\label{ques:ramos-sousa-24}
Regarding Theorem~\ref{thm:discrete-hardy}:
\begin{enumerate}
    \item[(i)] If $A = 1$, is there a constant $C_1 > 0$ such that, if $\Lambda, M$ satisfy~\eqref{eq:condition-discrete-hardy} and $f$ satisfies
    \[
        \sup_{\lambda\in\Lambda}|f(\lambda)|e^{\pi|\lambda|^2}+\sup_{\mu\in M}|\hat{f}(\mu)|e^{\pi|\mu|^2}<\infty,
    \]
    then $f(x) = c \, e^{-\pi x^2}$ for some $c \in \mathbb{C}$?

    \item[(ii)] What is the best constant $C_A > 0$ for $A>1$ such that if $\Lambda, M$ satisfy~\eqref{eq:condition-discrete-hardy}, then we can still conclude that $f \equiv 0$ in Theorem~\ref{thm:discrete-hardy}?
\end{enumerate}
\end{question}

We will resolve part~(ii) of Question~\ref{ques:ramos-sousa-24} by showing that Theorem~\ref{thm:discrete-hardy} holds with $C_A=2$. For $A=1$, we show that the statement in part~(i) of Question~\ref{ques:ramos-sousa-24} holds with $C_1=2$.
\section{Main results}
The driving engine of our results will be the following decay transfer principle.
\begin{theorem}\label{thm:discrete-beurling}
    Let $a,b>0$, and let $\Lambda,M\subset\mathbb{R}$ be a supercritical pair with parameters $(p,q)$. Assume that
    \begin{equation}\label{eq:discrete-beurling-condition}
        \limsup_{j\to\pm\infty}|\lambda_j|^{p-1}(\lambda_{j+1}-\lambda_j)<\frac{1}{2}\left(\frac{b}{a}\right)^\frac{1}{q}\quad\text{and}\quad\limsup_{j\to\pm\infty}|\mu_j|^{q-1}(\mu_{j+1}-\mu_j)<\frac{1}{2}\left(\frac{a}{b}\right)^\frac{1}{p}.
    \end{equation}
    If $f\in\mathcal{H}$ and
    \begin{equation}\label{eq:discrete-beurling-condition-2}
        \sup_{\lambda\in\Lambda}|f(\lambda)|(1+|\lambda|)^{-K}e^{a\pi|\lambda|^p}+\sup_{\mu\in M}|\hat{f}(\mu)|(1+|\mu|)^{-K}e^{b\pi|\mu|^q}<\infty
    \end{equation}
    for some $K>0$, then there exists a $\tilde{K}>0$ such that
    \begin{equation}\label{eq:discrete-beurling-conclusion}
        \sup_{x\in\mathbb{R}}|f(x)|(1+|x|)^{-\tilde{K}}e^{a\pi|x|^p}+\sup_{\xi\in\mathbb{R}}|\hat{f}(\xi)|(1+|\xi|)^{-\tilde{K}}e^{b\pi|\xi|^q}<\infty.
    \end{equation}
\end{theorem}
By the non-uniqueness part of Theorem~\ref{thm:kns-uniqueness}, Theorem~\ref{thm:discrete-beurling} fails if any $\varepsilon>0$ is added to the right-hand side of either condition in~\eqref{eq:discrete-beurling-condition}. Thus, the density condition is sharp.

\begin{remark}
Suppose that $\Lambda$ and $M$ satisfy~\eqref{eq:discrete-beurling-condition} and there exist constants $c_1,c_2,c_3,c_4>0$ and complete interpolating sequences
$\Gamma_1,\Gamma_2,\Gamma_3,\Gamma_4\subset\mathbb R$ for some Paley--Wiener
space such that
\[
\Lambda\cap[0,\infty)
=
\{c_1\gamma^\frac{1}{p}:\gamma\in\Gamma_1\cap[0,\infty)\},
\qquad
\Lambda\cap(-\infty,0]
=
\{-c_2\gamma^\frac{1}{p}:\gamma\in\Gamma_2\cap[0,\infty)\},
\]
and
\[
M\cap[0,\infty)
=
\{c_3\gamma^\frac{1}{q}:\gamma\in\Gamma_3\cap[0,\infty)\},
\qquad
M\cap(-\infty,0]
=
\{-c_4\gamma^\frac{1}{q}:\gamma\in\Gamma_4\cap[0,\infty)\}.
\]
Then one may take $\tilde{K}=K$ in~\eqref{eq:discrete-beurling-conclusion}, provided that~\eqref{eq:discrete-beurling-condition-2} is replaced with
\begin{equation}\label{eq:discrete-beurling-condition-3}
    \sum_{\lambda\in\Lambda}|f(\lambda)|(1+|\lambda|)^{-K}e^{a\pi|\lambda|^p}+\sum_{\mu\in M}|\hat{f}(\mu)|(1+|\mu|)^{-K}e^{b\pi|\mu|^q}<\infty.
\end{equation}
\end{remark}

We now get the following discrete generalization of Hardy's uncertainty principle.
\begin{corollary}\label{cor:discrete-hardy}
    Let $\Lambda,M\subset\mathbb{R}$ be a supercritical pair satisfying
    \[
    \max\left(\limsup_{j\to\pm\infty}|\lambda_j|(\lambda_{j+1}-\lambda_j),\limsup_{j\to\pm\infty}|\mu_j|(\mu_{j+1}-\mu_j)\right)<\frac{1}{2}.
    \]
    If $f\in\mathcal{H}$ and
    \[
    \sup_{\lambda\in\Lambda}|f(\lambda)|(1 + |\lambda|)^{-N}e^{\pi|\lambda|^2}+\sup_{\mu\in M}|\hat{f}(\mu)|(1 + |\mu|)^{-N}e^{\pi|\mu|^2}<\infty
    \]
    for some $N\geq 0$, then $f(x) = P(x) e^{-\pi x^{2}}$, where $P$ is a polynomial with $\deg P\leq N$.
\end{corollary}
\begin{proof}
    By Theorem~\ref{thm:discrete-beurling}, we know that
        \[
    \sup_{x\in\mathbb{R}}|f(x)|(1+|x|)^{-\tilde K}e^{\pi|x|^2}+\sup_{\xi\in\mathbb{R}}|\hat{f}(\xi)|(1+|\xi|)^{-\tilde K}e^{\pi|\xi|^2}<\infty
    \]
    for some $\tilde K>0$. By Theorem~\ref{thm:hardy}, this implies that $f(x)=P(x)e^{-\pi x^2}$ for some polynomial $P$ with $\deg P\leq \tilde K$. In particular,
    \[
    \sup_{\lambda\in\Lambda}|f(\lambda)|(1 + |\lambda|)^{-N}e^{\pi|\lambda|^2}=\sup_{\lambda\in\Lambda}(1 + |\lambda|)^{-N}|P(\lambda)|<\infty
    \]
    by the original hypothesis. Since $\Lambda$ is unbounded, this is only possible if $P$ has degree at most $N$.
\end{proof}
\medskip
Taking $N=0$ in Corollary~\ref{cor:discrete-hardy} resolves part~(i) of Question~\ref{ques:ramos-sousa-24}. For $A>1$, the same conclusion with the stronger weight $e^{A\pi |x|^2}$ forces $P\equiv 0$, resolving part~(ii) with the optimal constant $C_A=2$. Thus the optimal constant is $C_A=2$ for all $A\geq 1$.

We can use the same strategy to get a discrete version of Morgan's uncertainty principle.
\begin{theoremx}[Morgan's uncertainty principle~\cite{MR1574180}]\label{thm:morgan}
    Let $p,q>1$ and $\frac{1}{p}+\frac{1}{q}=1$. Suppose that $f\in L^1(\mathbb{R})$ and that $|f(x)|=O(e^{-\frac{2}{p}A\pi|x|^p})$ and $|\hat{f}(\xi)|=O(e^{-\frac{2}{q}A\pi|\xi|^q})$ for some $A>|\cos(\frac{r\pi}{2})|^\frac{1}{r}$ where $r:=\min\{p,q\}$. Then $f(x)\equiv0$.
\end{theoremx}
Applying Morgan's uncertainty principle together with Theorem~\ref{thm:discrete-beurling} yields the following result.
\begin{corollary}\label{cor:discrete-beurling}
    Let $\Lambda,M\subset\mathbb{R}$ be a supercritical pair satisfying
    \begin{align*}
        \limsup_{j\to\pm\infty}|\lambda_j|^{p-1}(\lambda_{j+1}-\lambda_j)<\frac{1}{2}\left(\frac{p}{q}\right)^{\frac{1}{q}}\quad\text{and}\quad\limsup_{j\to\pm\infty}|\mu_j|^{q-1}(\mu_{j+1}-\mu_j)<\frac{1}{2}\left(\frac{q}{p}\right)^{\frac{1}{p}}
    \end{align*}
    where $p,q>1$ and $\frac{1}{p}+\frac{1}{q}=1$. If $A>|\cos(\frac{r\pi}{2})|^\frac{1}{r}$ where $r:=\min\{p,q\}$, then there are no non-zero functions $f\in\mathcal{H}$ such that 
    \[
    \sup_{\lambda\in\Lambda}|f(\lambda)|e^{\frac{2}{p}A\pi|\lambda|^p}+\sup_{\mu\in M}|\hat{f}(\mu)|e^{\frac{2}{q}A\pi|\mu|^q}<\infty.
    \]
\end{corollary}
\begin{remark}
    Theorem~\ref{thm:discrete-beurling} will also be used as a key input in a forthcoming work~\cite{lysen-discrete-pauli} answering Question~23 of Ramos and Sousa~\cite{ramos2025pauli} on discrete
    Pauli pairs.
\end{remark}

We give a short overview of the proof strategy of Theorem~\ref{thm:discrete-beurling}. First, the assumptions imply that $f$ belongs to a Gelfand--Shilov class, so both $f$ and $\hat f$ extend to entire functions with finite order. Second, we build an interpolating correction term which matches $f$ on a suitable subsequence; after subtracting this correction, the prescribed small values may be treated as zeros while the Phragm\'en--Lindel\"of indicator remains controlled. Third, a sectorial zero-counting argument gives inequalities between the real-line decay indicators of $f$ and $\hat f$. The density assumptions cause a contradiction between these inequalities unless $f$ and $\hat{f}$ have the desired real-line decay. Finally, a Phragm\'en--Lindel\"of argument converts this into the global polynomially weighted bounds.
\section{Auxiliary results on entire functions}
This section collects the entire-function tools needed in the proof.

We say that $f$ belongs to the Gelfand--Shilov space $\mathcal{S}(p,q)$ if
\[
\sup_{x\in\mathbb{R}}|f(x)|e^{c|x|^p}+\sup_{\xi\in\mathbb{R}}|\hat{f}(\xi)|e^{c|\xi|^q}<\infty
\]
for some $c>0$~\cite{gelfandshilovfourier, gelfandshilov}. By standard estimates of the Fourier transform, if $f\in\mathcal{S}(p,q)$, where $\frac{1}{p}+\frac{1}{q}=1$ and $1<p,q<\infty$, then $f$ extends to an entire function of order $p$ and $\hat{f}$ extends to an entire function of order $q$.

We will use the following lemma to show that $f$ must belong to a Gelfand--Shilov class under the assumptions of Theorem~\ref{thm:discrete-beurling}. This allows us to study the problem using results from entire function theory.
\begin{lemmax}[Proposition~22 in~\cite{ramos2025pauli}]\label{lem:gelfand-shilov}
    Let $(\Lambda, M)$ be a supercritical pair with parameters $(p,q)$. Suppose also that $f$ is in $\mathcal{H}$ and satisfies
    \[
    \sup_{\lambda\in\Lambda}|f(\lambda)|e^{a|\lambda|^p}+\sup_{\mu\in M}|\hat{f}(\mu)|e^{a|\mu|^q}<\infty
    \]
    for some $a>0$. Then there exists $c_a>0$ such that
    \[
    \sup_{x\in\mathbb{R}}\left(|f(x)|e^{c_a|x|^p}+|\hat{f}(x)|e^{c_a|x|^q}\right)<\infty.
    \]
\end{lemmax}
The proof of Lemma~\ref{lem:gelfand-shilov} uses Poincar\'e--Wirtinger type estimates similar to those found in~\cite{kulikov2025,lysen2025criticalasymmetricfourieruniqueness}.

Using methods from the construction of interpolating functions in the Paley--Wiener space, we can treat the points where $|f|$ and $|\hat{f}|$ are small as zeros after subtracting suitable interpolating functions.

It is well known in entire function theory that the density of zeros is closely related to the asymptotic growth of the function. This can be represented using the Phragm\'en--Lindel\"of indicator of the function. The Phragm\'en--Lindel\"of indicator $h_f$ of an entire function $f$ of order $\rho>0$ is defined as
\[
h_f(\theta):=\limsup_{r\to\infty}\frac{\log|f(re^{i\theta})|}{r^\rho}.
\]
Since our functions $f$ and $\hat{f}$ extend to entire functions, we can study their Phragm\'en--Lindel\"of indicators.

Our interpolating function will be constructed using a canonical product. The canonical product for some sequence $(a_n)_{n=1}^\infty\subset\mathbb{C}$ is defined as
\[
\Pi(z):=\prod_{n=1}^\infty G\left(\frac{z}{a_n},k\right)
\]
where $G$ denotes the Hadamard canonical factors
\[
G(z,k):=\begin{cases}
    1-z,\quad&k=0,\\
    (1-z)e^{z+\frac{z^2}{2}+\cdots+\frac{z^k}{k}},\quad&k>0
\end{cases}
\]
and $k$ is the smallest non-negative integer such that
\[
\sum_{n=1}^\infty\frac{1}{|a_n|^{k+1}}<\infty.
\]
We now recall estimates for canonical products where $k<\rho<k+1$.
\begin{definitionx}
    A set of disks $(C_j)$ in the complex plane is called a $C^0$-set if
    \[
    \lim_{R\to\infty}\frac{1}{R}\sum_{|z_j|<R} r_j=0
    \]
    where $z_j$ are the centers of $(C_j)$, and $r_j$ are their radii.
\end{definitionx}
\begin{theoremx}[Theorem 7, Chapter 12~\cite{levinentirefunctions}]\label{thm:canonical-product-estimate}
Let $\Gamma:=(\gamma_n)_{n=1}^\infty\subset[0,\infty)$, and let $\Pi$ be the canonical product of $\Gamma$. Assume that $\rho$ is not an integer and that the density
\[
\Delta:=\lim_{r \to \infty} \frac{|\Gamma\cap[0,r]|}{r^{\rho}},
\]
exists. Then there exists a $C^0$-set of disks $(C_j)$ outside which the asymptotic relation
\begin{equation*}
\log \left| \Pi\!\left(r e^{i\theta}\right) \right|
= \frac{\pi \Delta}{\sin \pi \rho} \, r^{\rho} \cos \rho(\theta - \pi) + o(r^{\rho}), \quad r \to \infty,
\end{equation*}
holds uniformly with respect to $\theta\in[0,2\pi]$.
\end{theoremx}

We can use Theorem~\ref{thm:canonical-product-estimate} to estimate canonical products of order $1$ with real zeros that are symmetric about the origin.
\begin{lemma}\label{lem:canonical-product-estimate}
    If $\Gamma:=(\gamma_n)_{n=1}^\infty$ is a sequence of positive numbers such that $\Delta:=\displaystyle{\lim_{n\to\infty}\frac{n}{\gamma_n}}$, and if
    \[
    \Pi(z):=\prod_{n=1}^\infty\left(1-\frac{z^2}{\gamma_n^2}\right),
    \]
    then there exists a $C^0$-set of disks $(D_j)$, centered at the zeros $\pm\gamma_j$, outside of which the relation
    \[
    \log|\Pi(re^{i\theta})|=\Delta\pi r|\sin(\theta)|+o(r)\quad\text{as $r\to\infty$},
    \]
    holds uniformly with respect to $\theta\in[0,2\pi]$.
\end{lemma}
\begin{proof}
    The sequence $(\gamma_n^2)_{n=1}^\infty$ has density $\Delta$ with respect to the order $\rho=\frac{1}{2}$, since
    \[
    \lim_{n\to\infty}\frac{n}{(\gamma_n^2)^\rho}=\lim_{n\to\infty}\frac{n}{\gamma_n}=\Delta.
    \]
    Let
     \[
    \Phi(w):=\prod_{n=1}^\infty\left(1-\frac{w}{\gamma_n^2}\right).
    \]
    Since $0<\rho<1$, the canonical factors have genus $0$, and $\Phi$ is the canonical product associated with $(\gamma_n^2)$. Applying Theorem~\ref{thm:canonical-product-estimate} to $\Phi$ gives
    \[
    \log|\Phi(Re^{i\phi})|=\pi\Delta R^{1/2}\cos\left(\frac{\phi-\pi}{2}\right)+o(R^{1/2})
    \]
    uniformly in $\phi$, outside a $C^0$-set. Since $\Pi(z)=\Phi(z^2)$, putting $z=re^{i\theta}$ gives $R=r^2$ and $\phi=2\theta$ modulo $2\pi$. Therefore
    \[
    \log|\Pi(re^{i\theta})|=\Delta\pi r|\sin(\theta)|+o(r)\quad\text{as $r\to\infty$},
    \]
    uniformly in $\theta\in[0,2\pi]$, outside the corresponding exceptional set in the complex plane. This exceptional set is again a $C^0$-set and may be taken to consist of disks centered at the zeros $\pm\gamma_j$.
\end{proof}
\section{Constructing the correction}
Our goal in this section is to construct an interpolating function $g$ which matches $f$ on a certain discrete set while having a suitable Phragm\'en--Lindel\"of indicator. We can then later take the difference $f-g$ and apply zero counting arguments from entire function theory to prove Theorem~\ref{thm:discrete-beurling}.

We will also need another standard estimate for canonical products for certain well-behaved uniformly discrete sequences. A sequence $\Lambda=(\lambda_{j})_{j\in\mathbb{Z}}\subset\mathbb{R}$ is called uniformly discrete if $\displaystyle{\inf_{i\neq j}|\lambda_i-\lambda_j|>0}$. This estimate was originally used by Beurling to construct interpolating functions in the Paley--Wiener space.
\begin{lemmax}[Lemma 7, p.~357 in~\cite{MR1057614}]\label{lem:beurling-product}
Let $\Lambda=(\lambda_j)_{j\in\mathbb{Z} }\subset \mathbb{R}$ be a uniformly discrete sequence such that every interval of length $L$ in a subdivision of $\mathbb{R}$ contains exactly $m$ points of $\Lambda$. Assume $0 \in \Lambda$, and let $\displaystyle{\delta:=\inf_{i\neq j}|\lambda_i-\lambda_j|>0}$. The limit
\[
f(z) = \lim_{R \to \infty} \left\{ \prod_{\substack{|\lambda| < R \\ \lambda \neq 0}} \left(1 - \frac{z}{\lambda}\right) \right\}
\]
exists for all $z \in \mathbb{C}$. Moreover, $f$ is an entire function that vanishes on $\Lambda \setminus \{0\}$, satisfies $f(0) = 1$, and obeys
\begin{equation*}
|f(x + iy)| \leq C(|z| + 1)^{5m} e^{\pi \nu |y|}
\end{equation*}
for every $\nu>\frac{m}{L}$, where $C$ depends on $L,m,\nu$ and $\delta$.
\end{lemmax}
We can now use Lemma~\ref{lem:beurling-product} to construct the following interpolation result.
\begin{lemma}\label{lem:interpolation}
 Let $T=(t_k)_{k=1}^\infty\subset(0,\infty)$ be uniformly discrete. Assume that there are $L>0$ and $M\in\mathbb{N}$ such that every interval $I_n=[nL,(n+1)L)$ contains exactly $M$ points of $T$ for suffiently large $n$. Then for every $(\eta_k)_{k\in\mathbb{N}}\in\ell^\infty$, there exist $K,C>0$ and an entire function $g$ such that $g(t_k)=\eta_k$ for all $k\in\mathbb{N}$ and
 \[
 |g(re^{i\theta})|\leq C\delta^{-1}(1+r)^K|\Pi(re^{i\theta})|\quad\text{for all $\theta\in\left(-\frac{\pi}{2},\frac{\pi}{2}\right)$ and $r>0$}
 \]
 where $\Pi$ is the canonical product with zero set $\widetilde{\Gamma}:=\{\pm t_k:k\in\mathbb{N}\}$ and $\delta:=\inf_{\gamma\in\widetilde{\Gamma}}|re^{i\theta}-\gamma|$. Additionally, for some $n\geq 0$,
 \[
 |\Pi(x)|\lesssim (1+|x|)^n\quad \text{for all $x\in\mathbb{R}$}\quad\text{and}\quad 
 |g(x)|\lesssim(1+|x|)^n\quad\text{for all $x\in\mathbb{R}$.}
 \]
\end{lemma}
 \begin{proof}
    Let $\widetilde{\Gamma}:=\{\pm t_k:k\in\mathbb{N}\}$. Reflecting the subdivision to the negative axis, $\widetilde{\Gamma}$ has exactly $M$ points in each interval of a length $L$ subdivision of $\mathbb{R}$, up to finitely many intervals near the origin. These finite exceptions only change the products below by polynomial factors. Since their number is fixed and $\widetilde{\Gamma}$ is uniformly discrete, the constants below can be chosen independently of the shift.

     Fix $\nu>\frac{M}{L}$.
     Fix $t'\in T$ and define the shifted product
     \[
     B_{t'}(z):=\prod_{\gamma\in\widetilde{\Gamma}\setminus\{t'\}}\left(1-\frac{z}{\gamma-t'
     }\right).
     \]
     Applying Lemma~\ref{lem:beurling-product} to $\widetilde{\Gamma}-t'$ away from the finitely many exceptional intervals, and absorbing the corresponding finite factors into the polynomial term, we know that there are constants $C,N_0>0$, independent of $t'$, such that
     \[
     |B_{t'}(z)|\leq C(1+|z|)^{N_0}e^{\pi \nu |\Im z|}.
     \]
     If
     \[
     \Pi(z):=\prod_{k=1}^{\infty}\left(1-\frac{z^2}{t_k^2}\right)
     \]
     is the canonical product of $\widetilde{\Gamma}$, then the corresponding interpolation function satisfies
     \[
     \frac{\Pi(z)}{\Pi'(t')(z-t')}=B_{t'}(z-t').
     \]
     In particular, fixing one point $t_0\in T$ gives
     \[
     |\Pi(z)|\leq C|\Pi'(t_0)|\,|z-t_0|(1+|z-t_0|)^{N_0}e^{\pi \nu|\Im z|}.
     \]
     Since $t_0$ is fixed, this implies
     \[
     |\Pi(z)|\lesssim (1+|z|)^{N_0+1}e^{\pi \nu|\Im z|}.
     \]
     In particular,
     \[
     |\Pi(x)|\lesssim (1+|x|)^{N_0+1}\quad\text{for all $x\in\mathbb R$}.
     \]

     This means that we have the estimate
     \[
     \frac{|\Pi(z)|}{|\Pi'(t')||z-t'|}\leq C(1+|z-t'|)^{N_0}e^{\pi \nu|\Im z|}.
     \]
     Evaluating at $z=i$, we see that
     \[
     \frac{|\Pi(i)|}{|\Pi'(t')||i-t'|}\leq C(2+t')^{N_0}e^{\pi \nu}.
     \]
     Rearranging, we see that there exist $\tilde{C},K_0>0$ independent of $t'$ such that
     \[
     |\Pi'(t')|\geq \tilde{C}(1+t')^{-K_0}.
     \]
     Choose an integer $K>\max\{K_0,N_0\}+2$ and define
     \[
     g(z):=\sum_{k\in\mathbb{N}}\frac{(1+z)^K}{(1+t_k)^K}\frac{\Pi(z)}{\Pi'(t_k)(z-t_k)}\eta_k.
     \]
     The interpolation property gives $g(t_k)=\eta_k$. For $z=re^{i\theta}$, we have
     \begin{align*}
         |g(z)|&\leq|\Pi(re^{i\theta})|\|\eta\|_{\ell^\infty}\sum_{k\in\mathbb{N}}
        \frac{(1+r)^K}{(1+t_k)^K}\frac{1}{|\Pi'(t_k)||re^{i\theta}-t_k|}\\
         &\leq (1+r)^K|\Pi(re^{i\theta})|\|\eta\|_{\ell^\infty}\sum_{k\in\mathbb{N}}
        \frac{(1+t_k)^{K_0}}{\tilde{C}(1+t_k)^K|re^{i\theta}-t_k|}.
     \end{align*}
     Since $|re^{i\theta}-t_k|\geq \delta$, we know that
     \[
     |g(re^{i\theta})|\leq \delta^{-1} (1+r)^K|\Pi(re^{i\theta})|\left(\frac{\|\eta\|_{\ell^\infty}}{\tilde C}\sum_{k\in\mathbb{N}}\frac{(1+t_k)^{K_0}}{(1+t_k)^K}\right).
     \]
     Since $T$ has at most linear density and $K>K_0+2$, the sum converges:
     \[
     \sum_{k\in\mathbb{N}}\frac{(1+t_k)^{K_0}}{(1+t_k)^K}<\infty.
     \]

     For real $x$, the identity
     \[
     \frac{\Pi(x)}{\Pi'(t_k)(x-t_k)}=B_{t_k}(x-t_k)
     \]
     holds with the removable interpretation at $x=t_k$. Hence
     \begin{align*}
     |g(x)|&\leq C\|\eta\|_{\ell^\infty}(1+|x|)^K
     \sum_{k\in\mathbb{N}}\frac{(1+|x-t_k|)^{N_0}}{(1+t_k)^K}\\
     &\leq C\|\eta\|_{\ell^\infty}(1+|x|)^{K+N_0}
     \sum_{k\in\mathbb{N}}\frac{1}{(1+t_k)^{K-N_0}}.
     \end{align*}
     Since $K>N_0+2$ and $T$ has at most linear density, the last sum converges. Combining this with the preceding bound for $\Pi$, the two real-line estimates in the statement hold with
     $n=\max\{N_0+1,K+N_0\}$.
 \end{proof}
\section{An inequality between time and frequency real-line decay}
In this section we construct an inequality between asymptotic decay of $f$ and $\hat{f}$ given only the decay of $f$ on a discrete set. This will be crucial in our proof of Theorem~\ref{thm:discrete-beurling}.

Using Lemma~\ref{lem:interpolation}, we can treat the points in $\Lambda, M$ as zeros of $f$ and $\hat{f}$ respectively as long as they do not decay too fast.
We can then use a similar strategy to the one used in~\cite{kulikov2025}.

\begin{lemma}\label{lem:jensen-counting}
    Suppose that $0<\varphi<\frac{\pi}{2}$ and that $F$ is analytic in $S(\varphi):=\{z\in\mathbb{C}:|\arg(z)|<\varphi\}$, continuous up to the boundary, and satisfies $|F(z)|\leq C_F e^{a\pi|\Im(z)|}$ with some $a>0$. Let $\Gamma=(\gamma_j)_{j\in\mathbb{Z}}\subset\mathbb{R}_+$ be a discrete set such that, for some $\delta>0$ and $C_\Gamma>0$,
    \[
    |\Gamma\cap[u,v]|\geq a(1+\delta)(v-u)-C_\Gamma,\quad 0\leq u<v.
    \]
    If $F|_\Gamma=0$, then 
    \[
    \limsup_{x\to\infty}\frac{\log|F(x)|}{x}\leq -2a\delta\sin\varphi<0.
    \]
\end{lemma}
\begin{proof}
The argument is similar to the proof of Claim 5.3 in~\cite{kulikov2025}. Fix $x>0$ and put $r=x\sin\varphi$. The disk $x+r\mathbb{D}$ lies in $S(\varphi)$. If $F(x)=0$, the desired estimate is trivial, so assume $F(x)\neq0$. Let $n_x(t)$ be the number of zeros of $F$ in $x+t\mathbb{D}$, counted with multiplicity. Jensen's formula gives
\[
\log|F(x)|+\int_0^r\frac{n_x(t)}{t}\,dt
=\frac{1}{2\pi}\int_{-\pi}^\pi\log|F(x+re^{i\tau})|\,d\tau .
\]
The right-hand side satisfies
\[
\frac{1}{2\pi}\int_{-\pi}^\pi\log|F(x+re^{i\tau})|\,d\tau\leq2ar+O(1).
\]
Let $\rho=a(1+\delta)$. Since $F$ vanishes on $\Gamma$ and $x-r>0$, for $0<t<r$ we have
\[
n_x(t)\geq |\Gamma\cap[x-t,x+t]|\geq 2\rho t-C_\Gamma.
\]
With $A_0=\max\{1,\frac{C_\Gamma}{2\rho}\}$, it follows that
\[
\int_0^r\frac{n_x(t)}{t}\,dt
\geq \int_{A_0}^r\frac{2\rho t-C_\Gamma}{t}\,dt
=2\rho r+O(\log r).
\]
Combining these two estimates, we get
\[
\log|F(x)|\leq 2ar-2a(1+\delta)r+O(\log r)=-2a\delta x\sin\varphi+O(\log x).
\]
This means that 
\[
\limsup_{x\to\infty}\frac{\log|F(x)|}{x}\leq -2a\delta\sin\varphi<0.\qedhere
\] 
\end{proof}
In what follows, we will use the following notation:
\[
\kappa_f:=(2\pi)^{-1}\min\{-h_f(0),-h_f(\pi)\}\quad\text{and}\quad\kappa_{\hat{f}}:=(2\pi)^{-1}\min\{-h_{\hat{f}}(0),-h_{\hat{f}}(\pi)\}.
\]
\begin{lemmax}[Claim 5.1 in~\cite{kulikov2025}]\label{lem:kappa-estimate}
    Let $1<p,q<\infty$ with $\frac{1}{p}+\frac{1}{q}=1$. Assume that $\hat{f}$ is an entire function of order $q$ and that $\kappa_{\hat{f}}>0$. Then 
    $f$ is an entire function of order $p$ with Phragm\'{e}n--Lindel\"{o}f indicator $h_f(\theta)\leq2\pi\frac{|\sin \theta|^p}{p( q \kappa_{\hat{f}})^\frac{p}{q}}$.
\end{lemmax}
\begin{lemma}\label{lem:kappa-bound}
    Let $A>0$, let $f\in \mathcal{S}(p,q)$ where $\frac{1}{p}+\frac{1}{q}=1$, and assume that $0<\kappa_{f}<\frac{A}{2}$ and $\kappa_{\hat{f}}>0$. Assume that there exists a set $\Lambda=(\lambda_j)_{j\in\mathbb{Z}}\subset\mathbb{R}$ such that
    \[
    \limsup_{j\to\pm\infty}|\lambda_j|^{p-1}(\lambda_{j+1}-\lambda_j)<\alpha.
    \]
    If $\displaystyle{\sup_{\lambda\in\Lambda}|f(\lambda)|e^{A\pi|\lambda|^p}<\infty}$, then $\kappa_f>(2\alpha)^{-q}\kappa_{\hat{f}}$.
    
\end{lemma}
\begin{proof}
    We assume without loss of generality that the minimum in the definition of $\kappa_f$ is attained at the positive real axis, so that $h_f(0)=-2\pi\kappa_f$. If the minimum is attained at $\pi$, the same argument applies to $f(-x)$. We first prove that, for every $\theta\in(0,\frac{\pi}{2p})$,
    \[
    \kappa_f\geq (2\alpha)^{-1}\frac{\tan p\theta}{p}-\frac{1}{p(q\kappa_{\hat{f}})^\frac{p}{q}}\frac{|\sin\theta|^p}{\cos p\theta}.
    \]
    Suppose, to the contrary, that this estimate fails. Then there exist $\theta_0\in(0,\frac{\pi}{2p})$ and $0<\tilde{c}<1$ such that
    \[
    \kappa_f\cos p\theta_0+\frac{1}{p(q\kappa_{\hat{f}})^\frac{p}{q}}|\sin\theta_0|^p<(2\alpha)^{-1} \frac{\tilde{c}}{p}\sin p\theta_0.
    \]
    Consider the function
    \[
    G_\varepsilon(z)=e^{2\pi(\kappa_f-\varepsilon)z}f(z^\frac{1}{p}),
    \]
    where $z^\frac{1}{p}$ is taken with the principal branch in the right half-plane and $\varepsilon>0$ is small. Choosing $\varphi=p\theta_0$, Lemma~\ref{lem:kappa-estimate} and the Phragm\'{e}n--Lindel\"{o}f principle show that $|G_\varepsilon(z)|=O(e^{a_0\pi|\Im(z)|})$ throughout $S(\varphi)$, where $a_0=\tilde{c}(\alpha p)^{-1}<(\alpha p)^{-1}$.

    Since $\kappa_f<\frac{A}{2}$, we may choose $\sigma$ such that $2\kappa_f<\sigma<A$. Restricting to sufficiently large positive points of $\Lambda$ and using the mean value theorem, the strict limsup assumption gives
    \[
    \lambda_{j+1}^p-\lambda_j^p<\alpha p \quad\text{and}\quad \lambda_j>0
    \]
    for all $j\geq J$ for some $J\in\mathbb{Z}$. Now let $T=(\lambda_j^p)_{j=J}^\infty$. Choose $L>0$, $M\in\mathbb{N}$, and $h>0$ such that
    \[
        a_0<\frac{M}{L}<\frac{1}{\alpha p}
        \quad\text{and}\quad
        \alpha p<h<\frac{L}{M}.
    \]
    There exists $R_0>0$ such that every interval $[u,u+h]\subset[R_0,\infty)$ contains a point of $T$. Write $\ell=L/M$. For each $n$ sufficiently large such that $nL\geq R_0$ and each $s=0,\ldots,M-1$, choose one point of $T$ in $[nL+s\ell,nL+s\ell+h]$.
    
    This gives a uniformly discrete subsequence $T'=(\tau_j)$ with separation at least $\ell-h$ and with exactly $M$ points in each interval $[nL,(n+1)L)$ sufficiently far from the origin. Since $\sigma<A$, the values $f(\tau^\frac{1}{p})e^{\sigma\pi\tau}$ are bounded on $T'$. By Lemma~\ref{lem:interpolation}, applied to $T'$, there is an entire function $\tilde g$ such that
    \[
    \tilde g(\tau)=f(\tau^\frac{1}{p})e^{\sigma\pi\tau}\quad\text{for all $\tau\in T'$}.
    \]
    Let $g(z)=\tilde g(z)e^{-\sigma\pi z}$. Then $g(\tau)=f(\tau^\frac{1}{p})$ for $\tau\in T'$. By Lemma~\ref{lem:canonical-product-estimate}, we have the estimate
    \[
    \sup_{|\theta|\leq\omega}\limsup_{r\to\infty}\frac{\log|g(re^{i\theta})|}{r}\leq-2\pi(\kappa_f+\delta)    
    \]
    for some small angle $\omega>0$ and some $\delta>0$. We thus have
    \[
    |e^{2\pi(\kappa_f-\varepsilon)z}g(z)|\leq C e^{2\pi(\kappa_f-\varepsilon)|z|}e^{-2\pi(\kappa_f+\delta)|z|}= o(1)\quad\text{as $|z|\to\infty$}
    \]
    when $|\arg(z)|\leq \omega$. This means that in the angle $\min(\varphi,\omega)$, we have $|F_\varepsilon(z)|=O(e^{a_0\pi|\Im(z)|})$ for $F_\varepsilon(z)=e^{2\pi(\kappa_f-\varepsilon)z}(f(z^\frac{1}{p})-g(z))$.
    
    We now have $F_\varepsilon(\tau)=0$ for $\tau\in T'$. Since $T'$ contains exactly $M$ points in each interval which is sufficiently far from the origin and of the length-$L$ subdivision, and since $\frac{M}{L}>a_0$, applying Lemma~\ref{lem:jensen-counting} with $\delta_0=\frac{M}{La_0}-1>0$ gives
    \begin{equation}\label{eq:jensen-negative-bound}
    \limsup_{x\to\infty}\frac{\log|F_\varepsilon(x)|}{x}\leq -2a_0\delta_0\sin(\min(\varphi,\omega))<0
    \end{equation}
    independently of $\varepsilon>0$. Since the indicator of $e^{2\pi\kappa_f x}g(x)$ on the positive real axis is strictly smaller than that of $G_\varepsilon(x)$, the term $e^{2\pi(\kappa_f-\varepsilon)x}g(x)$ is negligible there. Consequently,
    \[
    \limsup_{x\to\infty}\frac{\log|G_\varepsilon(x)|}{x}=\limsup_{x\to\infty}\frac{\log|G_\varepsilon(x)-e^{2\pi(\kappa_f-\varepsilon)x}g(x)|}{x}=\limsup_{x\to\infty}\frac{\log|F_\varepsilon(x)|}{x}.
    \]
    On the other hand, the definition of $\kappa_f$ gives
    \[
    \limsup_{x\to\infty}\frac{\log|G_\varepsilon(x)|}{x}=-2\pi\varepsilon.
    \]
    Letting $\varepsilon\to0^+$ contradicts~\eqref{eq:jensen-negative-bound}. Hence, the estimate
    \[
    \kappa_f\geq(2\alpha)^{-1} \frac{\tan p\theta}{p}-\frac{1}{(q\kappa_{\hat{f}})^\frac{p}{q}}\frac{|\sin\theta|^p}{p\cos(p\theta)},\quad0<\theta<\frac{\pi}{2p}
    \]
    must hold. Applying the trigonometric inequality
    \begin{equation}\label{eq:trig-ineq}
            \frac{1}{p}\tan p\theta>\frac{\sin\theta}{(\cos p\theta)^\frac{1}{p}},\quad 0<\theta<\frac{\pi}{2p}
    \end{equation}
    and setting $\eta=\frac{\sin\theta}{(\cos p\theta)^\frac{1}{p}}$, we get
    \[
    \kappa_f>(2\alpha)^{-1}\eta-\frac{1}{(q\kappa_{\hat{f}})^\frac{p}{q}}\frac{\eta^p}{p}.
    \]
    Since $\eta$ ranges over all of $(0,\infty)$ as $\theta$ ranges over $(0,\frac{\pi}{2p})$, we may choose $\eta=q\kappa_{\hat{f}}(2\alpha)^{1-q}$. We then get
\begin{align*}
        \kappa_f&>(2\alpha)^{-1} q\kappa_{\hat{f}}(2\alpha)^{1-q}-\frac{1}{(q\kappa_{\hat{f}})^\frac{p}{q}}\frac{(q\kappa_{\hat{f}}(2\alpha)^{1-q})^p}{p}= q\kappa_{\hat{f}}(2\alpha)^{-q}-\frac{(2\alpha)^{-q}q\kappa_{\hat{f}}}{p}=(2\alpha)^{-q}\kappa_{\hat{f}}.
\end{align*}
    Hence $\kappa_f>(2\alpha)^{-q}\kappa_{\hat{f}}$.
\end{proof}
\section{Proof of Theorem~\ref{thm:discrete-beurling}}
\begin{proof}[{Proof of Theorem~\ref{thm:discrete-beurling}}]
    Let $\alpha=\frac{1}{2}(\frac{b}{a})^\frac{1}{q}$ and $\beta=\frac{1}{2}\left(\frac{a}{b}\right)^\frac{1}{p}$. By Lemma~\ref{lem:gelfand-shilov}, we know that $f\in\mathcal{S}(p,q)$.
    
    Assume now that $\kappa_f>(2\alpha)^{-q}\kappa_{\hat{f}}$ and $\kappa_{\hat{f}}>(2\beta)^{-p}\kappa_f$. Combining them yields $\kappa_f>(2\alpha)^{-q}(2\beta)^{-p}\kappa_f$. So, $2^{p+q}\alpha^q\beta^p=(2\alpha)^q(2\beta)^p>1$. For H\"older conjugates $p,q$, we know that $p+q=pq$, so $2\alpha^\frac{1}{p}\beta^\frac{1}{q}>1$. This is a contradiction since $\alpha^\frac{1}{p}\beta^\frac{1}{q}\leq\frac{1}{2}$. Hence, we must have either $\kappa_f\leq(2\alpha)^{-q}\kappa_{\hat{f}}$ or $\kappa_{\hat{f}}\leq(2\beta)^{-p}\kappa_f$.
    
    Without loss of generality, assume that $\kappa_f\leq(2\alpha)^{-q}\kappa_{\hat{f}}$. If $\kappa_f<\frac{a}{2}$, choose $a'<a$ such that $2\kappa_f<a'$. The polynomial factor in the hypothesis is absorbed by lowering the exponential constant from $a$ to $a'$, so
    \[
    \sup_{\lambda\in\Lambda}|f(\lambda)|e^{a'\pi|\lambda|^p}<\infty.
    \]
    Lemma~\ref{lem:kappa-bound} applied with $A=a'$ gives the contradiction $\kappa_f>(2\alpha)^{-q}\kappa_{\hat{f}}$. Hence, we know that $\kappa_f\geq \frac{a}{2}$. 
    
    Assume now that $\kappa_{\hat{f}}<\frac{b}{2}$. Choose $b'<b$ such that $2\kappa_{\hat{f}}<b'$. Again, the polynomial factor in the hypothesis is absorbed by lowering the exponential constant from $b$ to $b'$, so
    \[
    \sup_{\mu\in M}|\hat f(\mu)|e^{b'\pi|\mu|^q}<\infty.
    \]
    Applying Lemma~\ref{lem:kappa-bound} to $\hat f$ with parameters $(q,p)$ gives $\kappa_{\hat{f}}>(2\beta)^{-p}\kappa_f$, so $\kappa_{\hat{f}}>(2\beta)^{-p}\kappa_f\geq \frac{a}{2}(2\beta)^{-p}=\frac{b}{2}$, which is a contradiction. So, we also have $\kappa_{\hat{f}}\geq\frac{b}{2}$.
    By the definition of $\kappa_f$ and $\kappa_{\hat{f}}$, we now know that 
\[
h_f(0),\, h_f(\pi)\leq-a\pi\quad\text{and}\quad h_{\hat f}(0),\, h_{\hat f}(\pi)\leq-b\pi.
\]
Combining this with Lemma~\ref{lem:kappa-estimate}, we see that $h_f(\theta)\leq\frac{2\pi|\sin\theta|^p}{p(qb/2)^\frac{p}{q}}$. We now want to show that there exists a $\tilde\theta\in(0,\frac{\pi}{2p})$ such that 
\begin{equation}\label{eq:psi-indicator-bound}
    \frac{2\pi|\sin\tilde\theta|^p}{p(qb/2)^\frac{p}{q}}<-a\pi\cos(p\tilde\theta)+\frac{1}{\alpha p}\pi\sin(p\tilde\theta).
\end{equation}
Rearranging, we see that this is equivalent to
\begin{equation}\label{eq:psi-indicator-bound-rearranged}
    \frac{a}{2}<(2\alpha)^{-1} \frac{\tan p\tilde\theta}{p}-\frac{1}{(qb/2)^\frac{p}{q}}\frac{|\sin\tilde\theta|^p}{p\cos(p\tilde\theta)}.
\end{equation}
Letting $\eta:=\frac{\sin \tilde\theta}{(\cos p\tilde\theta)^\frac{1}{p}}$ and applying \eqref{eq:trig-ineq}, we see that
\begin{equation}\label{eq:indicator-bound-eta}
    (2\alpha)^{-1} \frac{\tan p\tilde\theta}{p}-\frac{1}{(qb/2)^\frac{p}{q}}\frac{|\sin\tilde\theta|^p}{p\cos(p\tilde\theta)}>(2\alpha)^{-1} \eta-\frac{\eta^p}{p(qb/2)^\frac{p}{q}}.
\end{equation}
Since $\eta$ can attain any value in $(0,\infty)$ depending on our choice of $\tilde{\theta}$, we can choose $\tilde\theta\in(0,\frac{\pi}{2p})$ such that $\eta=\frac{qb}{2(2\alpha)^{q-1}}$. Putting this into \eqref{eq:indicator-bound-eta}, we get the inequality
\begin{align*}
    (2\alpha)^{-1} \frac{\tan p\tilde\theta}{p}-\frac{1}{(qb/2)^\frac{p}{q}}\frac{|\sin\tilde\theta|^p}{p\cos(p\tilde\theta)}>(2\alpha)^{-1} \eta-\frac{\eta^p}{p(qb/2)^\frac{p}{q}}= \frac{a}{2}.
\end{align*}
This shows that~\eqref{eq:psi-indicator-bound-rearranged} holds and consequently we can pick $\tilde\theta\in(0,\frac{\pi}{2p})$ such that~\eqref{eq:psi-indicator-bound} holds. 

Choose $L>0$ and $M\in\mathbb{N}$ such that $d:=M/L<(\alpha p)^{-1}$ and $d$ is sufficiently close to $(\alpha p)^{-1}$ so that
\begin{equation}\label{eq:psi-indicator-bound-density}
    \frac{2\pi|\sin\tilde\theta|^p}{p(qb/2)^\frac{p}{q}}<-a\pi\cos(p\tilde\theta)+d\pi\sin(p\tilde\theta).
\end{equation}
Choose $h$ such that $\alpha p<h<L/M$.

Let $K_0$ be the exponent in the polynomial weight from the hypothesis, and choose an integer $m\geq K_0$. Restricting to sufficiently large positive points of $\Lambda$, write $T=(\lambda_j^p)$. The strict density assumption implies that every interval of length h sufficiently far from the origin contains a point of $T$. Writing $\ell=L/M$, choose one point of $T$ in each interval
\[
[nL+s\ell,nL+s\ell+h],
\quad s=0,\ldots,M-1,
\]
for each sufficiently large $n$. This gives a uniformly discrete subsequence $T'=(\tau_j)\subset T$ with separation at least $\ell-h$ and with exactly $M$ points in each interval which is sufficiently far from the origin $[nL,(n+1)L)$. The interpolation data
\[
\eta(\tau):=f(\tau^\frac{1}{p})e^{a\pi\tau}(1+\tau)^{-m},\quad \tau\in T',
\]
are bounded by the hypothesis on the points of $\Lambda$. By Lemma~\ref{lem:interpolation}, applied to $T'$, there is an entire function $g$ such that
\[
g(\tau)=f(\tau^\frac{1}{p})e^{a\pi\tau}(1+\tau)^{-m},\quad \tau\in T',
\]
and, for some $C,N>0$ and $n_0\geq0$,
\[
|g(re^{i\theta})|\leq \delta^{-1}C(1+r)^N|\Pi(re^{i\theta})|
\]
for all $\theta\in\left(-\frac{\pi}{2},\frac{\pi}{2}\right)$ and $r>0$ where $\delta:=\inf_{\tau\in T'}|re^{i\theta}-\tau|$, and
\[
|\Pi(x)|+|g(x)|\lesssim(1+|x|)^{n_0}\quad\text{for all $x\in\mathbb R$}.
\]
Consider now the corresponding canonical product $\Pi$. By Lemma~\ref{lem:canonical-product-estimate}, we know that the estimate
\[
    \log|\Pi(re^{i\theta})|=d\pi r|\sin(\theta)|+o(r)\quad\text{as $r\to\infty$},
    \]
holds uniformly with respect to $\theta\in[0,2\pi]$ outside a $C^0$-set of disks centered at the zeros of $\Pi$.
In the sector $|\arg z|< \frac{\pi}{2}$, $\Pi(z)=0$ precisely at the points of $T'$, and these zeros are simple. Hence the function
\[
\Psi(z):=\frac{f(z^\frac{1}{p})e^{a\pi z}-g(z)(1+z)^m}{\Pi(z)(1+z)^{N+m}}
\]
is holomorphic in the sector $|\arg z|<\frac{\pi}{2}$. By~\eqref{eq:psi-indicator-bound-density},
\[
\frac{|f(r^\frac{1}{p}e^{i\tilde{\theta}})e^{a\pi re^{ip\tilde{\theta}}}|}{|\Pi(re^{ip\tilde{\theta}})|(1+r)^{N+m}}=o(1)
\]
and the same holds on the ray $-p\tilde\theta$. We also know that
\[
\frac{|g(re^{ip\tilde{\theta}})|(1+r)^m}{|\Pi(re^{ip\tilde{\theta}})|(1+r)^{N+m}}=O(1).
\]
	Thus outside the $C^0$-set of disks the estimates above give
	\begin{equation}\label{eq:subexponential-bound}
	    |\Psi(z)|=O(e^{c|z|^{1+\varepsilon}})
	\end{equation}
	for every $c,\varepsilon>0$ in the closed sector $|\arg z|\leq p\tilde\theta$. We now use the usual proof of the Phragm\'en--Lindel\"of principle, but only along radii avoiding the exceptional disks. Since the exceptional disks form a $C^0$-set, there exists a sequence $R_k\to\infty$ such that the circular arcs
	\[
	\{R_ke^{i\theta}:|\theta|\leq p\tilde\theta\}
	\]
	do not meet any exceptional disk. Choose $1<\rho<\frac{\pi}{2p\tilde\theta}$, and take the branch of $z^\rho$ in the sector $|\arg z|<p\tilde\theta$. For $\varepsilon_0>0$, apply the maximum principle to $\Psi(z)e^{-\varepsilon_0 z^\rho}$ in the truncated sectors $\{z:|\arg z|<p\tilde\theta,\ |z|<R_k\}$. On the two radial sides the preceding estimates give a uniform bound, while on the circular arcs the bound~\eqref{eq:subexponential-bound} for $\Psi$ is dominated by $e^{-\varepsilon_0\operatorname{Re} z^\rho}$ when $k$ is large, after choosing the exponent in~\eqref{eq:subexponential-bound} smaller than $\rho$. Letting $k\to\infty$ and then $\varepsilon_0\to0^+$ shows that $\Psi$ is bounded in the sector $|\arg z|<p\tilde{\theta}$. Consequently, for all $x>0$,
	\[
	|f(x^\frac{1}{p})e^{a\pi x}-g(x)(1+x)^m|\lesssim |\Pi(x)|(1+x)^{N+m}.
	\]
	Using the real-line bounds from Lemma~\ref{lem:interpolation} and the triangle inequality,
	\begin{align*}
	    |f(x^\frac{1}{p})e^{a\pi x}|&\leq |f(x^\frac{1}{p})e^{a\pi x}-g(x)(1+x)^m|+|g(x)(1+x)^m|\\
	    &\lesssim(1+x)^{N+m+n_0}+(1+x)^{m+n_0}\\
	    &\lesssim (1+x)^{N+m+n_0}.
	\end{align*}
Repeating the same argument on the negative real axis, there exists $\tilde{K}\in\mathbb{N}$ such that
\[
|f(x)|=O((1+|x|)^{\tilde{K}}e^{-a\pi |x|^p}).
\]
Repeating the argument for $\hat{f}$ gives the desired bound.
\end{proof}
\section*{Acknowledgments}
I want to thank Kristian Seip for helpful comments and feedback on the exposition.

\bibliographystyle{plain}
\bibliography{references}

@misc{adve2023density,
  author = {Adve, A.},
  title = {Density criteria for {Fourier} uniqueness phenomena in $\mathbf{R}^d$},
  year = {2023},
  eprint = {2306.07475},
  archiveprefix = {arXiv},
  primaryclass = {math.CA},
  note = {\href{https://arxiv.org/abs/2306.07475}{arXiv:2306.07475}},
}

@article{ramos2025pauli,
  author = {Ramos, J. P. G. and Sousa, M.},
  title = {On {Pauli} pairs and {Fourier} uniqueness problems},
  journal = {J. Lond. Math. Soc. (2)},
  volume = {112},
  number = {6},
  pages = {Paper No. e70358, 29},
  year = {2025},
  doi = {10.1112/jlms.70358},
}

@article{Radchenko2019,
  author = {Radchenko, D. and Viazovska, M.},
  title = {{Fourier} interpolation on the real line},
  journal = {Publ. Math. Inst. Hautes \'Etudes Sci.},
  volume = {129},
  pages = {51--81},
  year = {2019},
  doi = {10.1007/s10240-018-0101-z},
}

@article{ramos2022fourier,
  author = {Ramos, J. P. G. and Sousa, M.},
  title = {{Fourier} uniqueness pairs of powers of integers},
  journal = {J. Eur. Math. Soc. (JEMS)},
  volume = {24},
  number = {12},
  pages = {4327--4351},
  year = {2022},
  doi = {10.4171/jems/1194},
}

@article{kulikov2025,
  author = {Kulikov, A. and Nazarov, F. and Sodin, M.},
  title = {{Fourier} uniqueness and non-uniqueness pairs},
  journal = {J. Math. Phys. Anal. Geom.},
  volume = {21},
  number = {1},
  pages = {84--130},
  year = {2025},
  doi = {10.15407/mag21.01.04},
}

@misc{lysen2025criticalasymmetricfourieruniqueness,
  author = {Lysen, T. K.},
  title = {Critical and asymmetric {Fourier} uniqueness pairs},
  year = {2025},
  eprint = {2509.17600},
  archiveprefix = {arXiv},
  primaryclass = {math.CA},
  note = {\href{https://arxiv.org/abs/2509.17600}{arXiv:2509.17600}},
}

@book{levinentirefunctions,
  author = {Levin, B. Ya.},
  title = {Lectures on entire functions},
  series = {Translations of Mathematical Monographs},
  volume = {150},
  note = {In collaboration with and with a preface by Yu.\ Lyubarskii,
          M. Sodin, and V. Tkachenko. Translated from the Russian manuscript by
          V. Tkachenko},
  publisher = {American Mathematical Society},
  address = {Providence, RI},
  pages = {xvi+248},
  year = {1996},
  doi = {10.1090/mmono/150},
}

@article{hardyuncertainty,
  author = {Hardy, G. H.},
  title = {A theorem concerning {Fourier} transforms},
  journal = {J. London Math. Soc.},
  volume = {8},
  number = {3},
  pages = {227--231},
  year = {1933},
  doi = {10.1112/jlms/s1-8.3.227},
}

@article{gelfandshilovfourier,
  author = {Chung, J. and Chung, S.-Y. and Kim, D.},
  title = {Characterizations of the {Gel'fand-Shilov} spaces via {Fourier}
           transforms},
  journal = {Proc. Amer. Math. Soc.},
  volume = {124},
  number = {7},
  pages = {2101--2108},
  year = {1996},
  doi = {10.1090/S0002-9939-96-03291-1},
}

@book{gelfandshilov,
  author = {Gel'fand, I. M. and Shilov, G. E.},
  title = {Generalized functions. {Vol}. 2},
  note = {Spaces of fundamental and generalized functions. Translated from the
          1958 Russian original by M. D. Friedman, A. Feinstein, and
          C. P. Peltzer. Reprint of the 1968 English translation},
  publisher = {AMS Chelsea Publishing},
  address = {Providence, RI},
  pages = {x+261},
  year = {2016},
  doi = {10.1090/chel/378},
}

@article{MR1574180,
  author = {Morgan, G. W.},
  title = {A note on {Fourier} transforms},
  journal = {J. London Math. Soc.},
  volume = {9},
  number = {3},
  pages = {187--192},
  year = {1934},
  doi = {10.1112/jlms/s1-9.3.187},
}

@book{MR1057614,
  author = {Beurling, A.},
  title = {The collected works of {Arne Beurling}. {Vol}. 2},
  series = {Contemporary Mathematicians},
  note = {Harmonic analysis. Edited by L. Carleson, P. Malliavin,
          J. Neuberger, and J. Wermer},
  publisher = {Birkh\"auser Boston, Inc.},
  address = {Boston, MA},
  pages = {xx+389},
  year = {1989},
}

@unpublished{lysen-discrete-pauli,
  author = {Lysen, T. K.},
  title = {Discrete {Pauli} pairs},
  note = {In preparation},
}
\end{document}